\documentclass[12pt]{article}

\usepackage[margin=1in]{geometry}
\usepackage{amsmath,amsthm,caption,graphicx,hyperref,amsfonts,xcolor,times}

\hypersetup{
  colorlinks = true,
  urlcolor = blue,
  linkcolor = blue,
  citecolor = blue
}

\usepackage[capitalise]{cleveref}

\usepackage[square,numbers]{natbib}

\newcommand\wh[1]{\widehat{#1}}
\newcommand\bs[1]{\boldsymbol{#1}}

\newtheorem{theorem}{Theorem}
\newtheorem{lemma}{Lemma}[section]
\newtheorem{proposition}[lemma]{Proposition}

\newtheorem{definition}{Definition}[section]

\newtheorem{remark}{Remark}[section]

\newcommand{\R}{\mathbb{R}}

\renewcommand{\P}{\mathbb{P}}
\newcommand{\E}{\mathbb{E}}
\newcommand{\V}{\mathbb{V}}
\newcommand{\X}{\mathbf{X}}
\DeclareMathOperator{\1}{\mathbf{1}}

\newcommand{\g}{\mathbf{g}}
\newcommand{\lse}{\mathrm{lse}}
\DeclareMathOperator{\Var}{\mathrm{Var}}
\DeclareMathOperator{\Cov}{\mathrm{Cov}}
\DeclareMathOperator{\Corr}{\mathrm{Corr}}

\DeclareMathSymbol{\shortminus}{\mathbin}{AMSa}{"39}

\title{Local--global correlations of dynamics\\ on disordered energy landscapes}

\author{
Jacob Calvert\thanks{Institute for Data Engineering and Science, Georgia Institute of Technology, Atlanta, GA 30332.} \:\!\textsuperscript{,}\! \thanks{Santa Fe Institute, Santa Fe, NM 87501.}
\and 
Dana Randall\thanks{School of Computer Science, Georgia Institute of Technology, Atlanta, GA 30332.\\ 
\indent \indent \negmedspace \textit{Email addresses}: \texttt{calvert@gatech.edu}, \,  \texttt{randall@cc.gatech.edu}.}
}

\date{}

\begin{document}

\maketitle

\begin{abstract}
The stationary distribution of a continuous-time Markov chain generally arises from a complicated global balance of probability fluxes. Nevertheless, empirical evidence shows that the effective potential, defined as the negative logarithm of the stationary distribution, is often highly correlated with a simple local property of a state---the logarithm of its exit rate. To better understand why and how typically this correlation is high, we study reversible reaction kinetics on energy landscapes consisting of Gaussian wells and barriers, respectively associated with the vertices and edges of regular graphs. We find that for the correlation to be high it suffices for the heights of the barriers to vary significantly less than the depths of the wells, regardless of the degree of the underlying graph. As an application, we bound below the expected correlation exhibited by dynamics of the random energy model, known as the Bouchaud trap model. We anticipate that the proof, which combines a general lower bound of the expected correlation with the Gaussian concentration inequality, can be extended to several other classes of models.
\end{abstract}

\section{Introduction}\label{sec:intro}

Continuous-time Markov chains play a fundamental role in stochastic thermodynamics, which seeks to extend the ideas of classical thermodynamics to systems that are ``small'' or ``far'' from equilibrium \cite{VAnDEnBROECK20156}. In this context, Markov chains primarily serve to model the mesoscopic dynamics of physical system, but can also describe the effective macroscopic dynamics that emerges from it \cite{RevModPhys.97.015002}. The connection between Markov chains and thermodynamics begins with the idea that a stationary Markov chain satisfying detailed balance is analogous to the steady state of a physical system in thermal equilibrium, partly because both are statistically time-reversible \cite{Zia_2007}. The stationary distributions of nonreversible Markov chains, which are instead maintained by a global balance of probability fluxes along cycles of states, correspondingly model nonequilibrium steady states. The key observation is that important thermodynamic properties of a system, like entropy production, can be expressed in terms of these cycle fluxes \cite{RevModPhys.48.571}. For this reason, many results of stochastic thermodynamics---including fluctuation theorems \cite{crooks1998,lebowitz1999,andrieux2007}, thermodynamic uncertainty relations \cite{PhysRevLett.114.158101,PhysRevLett.116.120601}, bounds on nonequilibrium response \cite{Owen2020,PhysRevLett.132.037101}, and force-flow relations \cite{yang2025nonequilibriumforceflowrelationsnetworks}---concern the way and extent to which Markov chains violate detailed balance. 

\subsection{Arrhenius parameterization}

With the connection between Markov chains and thermodynamics in mind, we express the transition rates in the form of a generalized Arrhenius law \cite{doi:10.1073/pnas.1918386117,Owen2020,PhysRevLett.133.227101,PhysRevLett.132.037101,PhysRevLett.134.157101,floyd2025localimperfectfeedbackcontrol}. To be precise, we consider a continuous-time Markov chain that is irreducible on a set of $n \geq 2$ states, which we denote by $[n] = \{1,\dots,n\}$. For simplicity, we conflate the Markov chain with its transition rate matrix, which we denote by $\mathbf{Q} = (Q_{ij})_{i,j \in [n]}$. We view its states as the vertices of an underlying adjacency graph $G$, an edge $(i,j)$ of which is present if and only if the corresponding transition rate $Q_{ij}$ is positive. We express these transition rates as
\begin{equation}\label{eq:rates_w_b_f}
Q_{ij} = \exp \left( W_i - B_{ij} +  F_{ij}\right), \quad (i,j) \in E(G),
\end{equation}
in terms of vertex parameters $W_i$, symmetric edge parameters $B_{ij} = B_{ji}$, and anti-symmetric edge parameters $F_{ij} = - F_{ji}$, all real numbers.\footnote{Note that this parameterization is not unique.} Here and throughout, $G$ is a simple graph with undirected edges and vertex set $[n]$.

As \cite{Owen2020} explains, the transition rates when cast in this form are reminiscent of the Arrhenius expression for the reaction rates of a physical system evolving in a landscape of energy wells $W_i$ and energy barriers $B_{ij}$, driven by forces $F_{ij}$ (\cref{fig1}a). While the analogy is imperfect, it is accurate in the following important respect. In the absence of forcing (i.e., $F_{ij}=0$ for all $i$ and~$j$), the Markov chain satisfies detailed balance and its stationary distribution $\bs{\pi} = (\pi_i)_{i \in [n]}$ has the Boltzmann form:
\begin{equation}\label{eq:boltzmann}
    \pi_i \propto \exp (-W_i).
\end{equation}
In reference to \cref{eq:boltzmann}, even when a Markov chain violates detailed balance, the quantity $-\log \pi_i$ is referred to as the effective potential of state $i$ \cite{strang_applications_2020}.

\begin{figure}
    \centering
    \includegraphics[width=\linewidth]{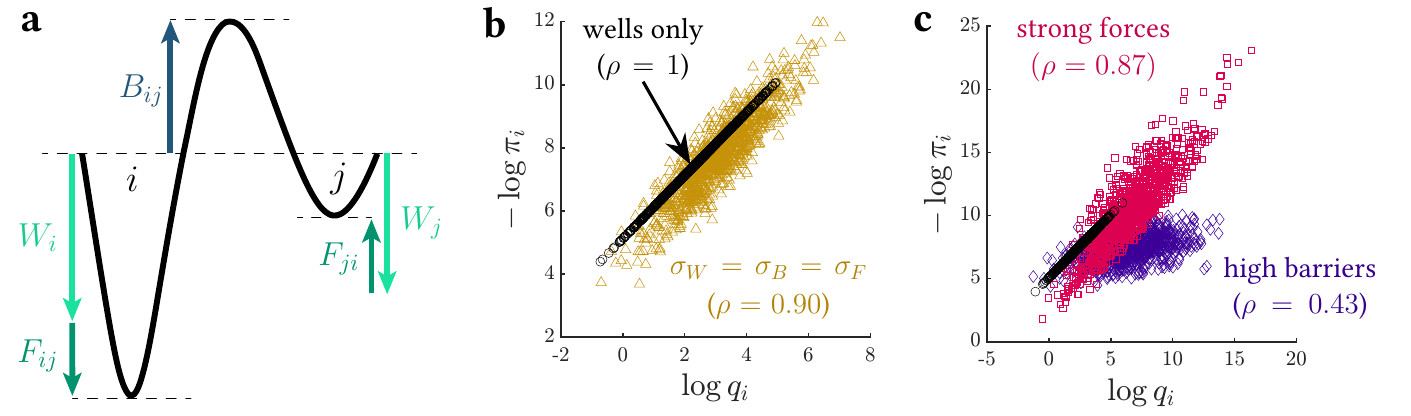}
    \caption{Local--global correlations in Arrhenius-like dynamics. (\textbf{a}) The transition rates in \cref{eq:rates_w_b_f} evoke the Arrhenius expression for the dynamics of a physical system in an energy landscape of wells $W_i$, separated by symmetric barriers $B_{ij} = B_{ji}$, and driven by anti-symmetric forces $F_{ij} = - F_{ij}$. (\textbf{b}) When these components are independent, centered normal random variables with standard deviations $\sigma_W = \sigma_B = \sigma_F = 1$, the effective potential $-\log \pi_i$ of the resulting Markov chain is highly collinear with $\log q_i$ (yellow triangles represent individual states). For comparison, in the absence of barriers and forces (i.e., $\sigma_B = \sigma_F = 0$), the effective potential $-\log \pi_i$ and the log exit rates $\log q_i$ are perfectly collinear because $\pi_i \propto 1/q_i$ (black circles). (\textbf{c}) Collinearity is greater in the presence of strong forces, which violate detailed balance ($\sigma_F = 4\sigma_W$, $\sigma_B = 0$; red squares), than in the presence of barrier heights of equal magnitude, which preserve it ($\sigma_B = 4 \sigma_W$, $\sigma_F = 0$; blue diamonds). In all cases, $\sigma_W = 1$, $n = 2^{10}$, the adjacency graph $G$ is the $10$-dimensional hypercube, and $\rho$ denotes the correlation coefficient.
    }
    \label{fig1}
\end{figure}

\subsection{Local--global correlations in Markov chains}\label{subsec: local global}

While the results of stochastic thermodynamics generally concern the way and extent to which detailed balance fails, an emerging perspective instead emphasizes the relationship  between local dynamical information and the global steady state distribution \cite{Chvykov2021,Calvert2024,calvert2025noteasymptoticuniformitymarkov}. In Markov chain terms, the key observation is that, in many cases, the effective potential of a state is highly collinear with the logarithm of its exit rate
\begin{equation*}
    q_i = \sum_{j:\, j \sim i} Q_{ij},
\end{equation*}
where $j \sim i$ indicates that $(i,j)$ is an edge of $G$ (\cref{fig1}b). We measure the collinearity of these quantities using their correlation for a uniformly random state $I \in [n]$:
\begin{equation}\label{eq:def_rho}
    \rho = \Corr (-\log \pi_I, \log q_I) = -\frac{\Cov (\log \pi_I, \log q_I)}{\sqrt{\Var (\log \pi_I) \Var (\log q_I)}}.
\end{equation}
Note that $\rho$ is well defined whenever $\pi_I$ and $q_I$ have positive, finite variances.

The effective potential $-\log\pi_i$ is a global property of a Markov chain, in the sense that estimating it from a trajectory generally requires the observation of multiple returns to state $i$, which can entail visits to faraway states (see \cite{nIPS2013_99bcfcd7}, for example). In contrast, to estimate $q_i$, it suffices to repeatedly initialize the Markov chain in state $i$ and average the times it takes to leave, because the exit rate $q_i$ is equal to the reciprocal of the expected duration of a visit to state $i$. In this sense, $\log q_i$ is a local property of a state, and $\rho$ measures a local--global correlation.

An intriguing aspect of $\rho$, especially in the context of stochastic thermodynamics, is that the extent of correlation is generally unrelated to whether a Markov chain satisfies detailed balance (\cref{fig1}c). It is therefore especially important to understand why and how typically systems exhibit high correlation. In \cite{Calvert2024}, we derived a formula for~$\rho$ in terms of two further Markov chain parameters, and empirically studied $\rho$ in Markov chains that arise in various scientific domains. However, no previous work has rigorously estimated~$\rho$ in a broad class of Markov chains.

\section{Main results}\label{sec:main result}

To better understand why high correlation arises, and how typically it does so, we fix a graph $G$ and analyze the correlation exhibited by Arrhenius-like dynamics on a disordered energy landscape consisting of wells and barriers alone (\cref{fig2}). Given well depths $\mathbf{W} = (W_i)_{i \in V(G)}$ and barrier heights $\mathbf{B} = (B_{ij})_{(i,j) \in E(G)}$ on $G$, we define $\mathbf{Q} = \mathrm{Arr}_G (\mathbf{W},\mathbf{B})$ to be the transition rate matrix with entries
\begin{equation}\label{eq:rates_w_b}
Q_{ij} = \exp(W_i - B_{ij}), \quad (i,j) \in E(G).
\end{equation}

\begin{definition}\label{def:landscape}
Let $\sigma_W$ and $\sigma_B$ be positive real numbers and let $G$ be a graph. We define a $(\sigma_W,\sigma_B)$-disordered energy landscape on $G$ to be a pair $(\mathbf{W},\mathbf{B})$ of i.i.d.\ $\mathcal{N}(0,\sigma_W^2)$ well depths $\mathbf{W} = (W_i)_{i \in V(G)}$, and barrier heights $\mathbf{B} = (B_{ij})_{(i,j) \in E(G)}$ that satisfy $B_{ij} = B_{ji} \sim \mathcal{N}(0,\sigma_B^2)$ for every $(i,j) \in E(G)$.
\end{definition}

While the barrier heights in \cref{fig2} are i.i.d.\ (aside from the symmetry condition $B_{ij} = B_{ji}$), our main results concern a broader class of disorder. Specifically, we allow the barrier heights adjacent to a state $i$ to depend on one another through the corresponding well depth. This flexibility is important for an application that we shortly present.

\begin{definition}\label{def:separable}
    We say that a disordered energy landscape $(\mathbf{W},\mathbf{B})$ on $G$ is separable if there exist $f: \R \to \R$ and $\sigma > 0$ such that, for every $i \in V(G)$, $(B_{ij} - f(W_i))_{j: \, j \sim i}$ is a vector of i.i.d.\ $\mathcal{N}(0,\sigma^2)$ random variables, which are also independent of $W_i$.
\end{definition}

The first of our main results explains that, if the typical barrier height in a separable energy landscape is small relative to the typical well depth, then the expected correlation $\rho$ exhibited by Arrhenius-like dynamics is close to~$1$. To ensure that $\rho$ is well defined, we stipulate that the exit rates $\mathbf{q}$ are a.s.\ non-identical, meaning that there are states $i$ and $j$ such that $q_i \neq q_j$.

\begin{figure}
    \centering
    \includegraphics[width=\linewidth]{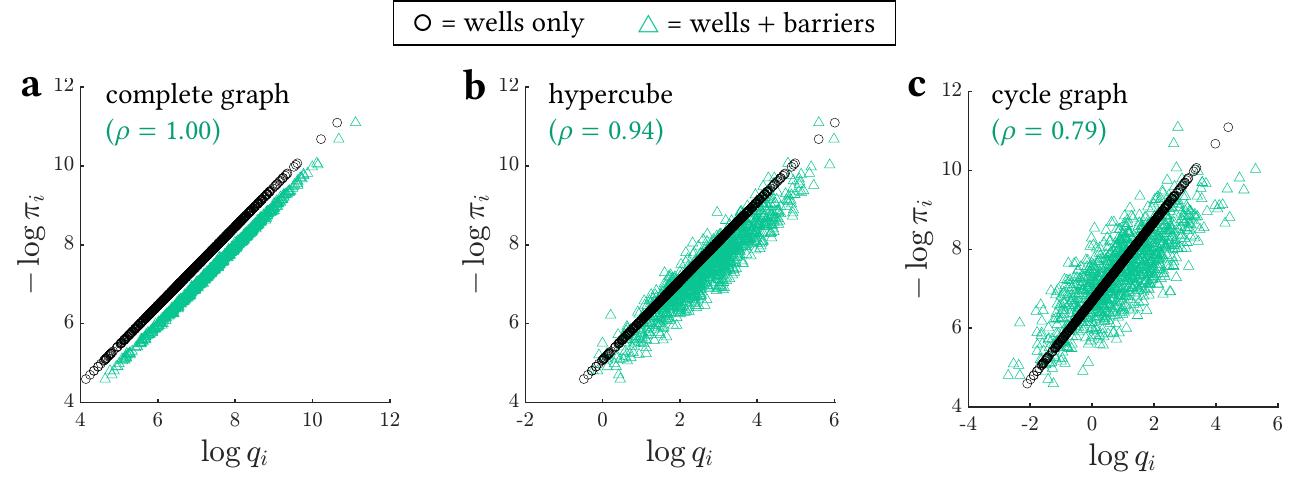}
    \caption{Local--global correlations in reversible Arrhenius-like dynamics on regular graphs. In the absence of barriers and forces ($\sigma_W = 1$ and $\sigma_B = \sigma_F = 0$, black circles), the effective potential ($-\log \pi_i$) and log exit rates ($\log q_i$) are perfectly collinear and $\rho = 1$. Collinearity approximately persists when the typical barrier height is equal to the typical well depth ($\sigma_W=\sigma_B=1$ and $\sigma_F = 0$, green triangles), when the adjacency graph $G$ is (\textbf{a}) the complete graph, (\textbf{b}) the hypercube, and (\textbf{c}) the cycle graph. In all cases, $n = 2^{10}$.
    }
    \label{fig2}
\end{figure}

\begin{theorem}\label{thm:rho}
Let $G$ be a connected, regular graph with $n \geq 6$ vertices, and let $\sigma_W$ and $\sigma_B$ be positive real numbers. If $(\mathbf{W},\mathbf{B})$ is a separable, $(\sigma_W,\sigma_B)$-disordered energy landscape on $G$ for which $\mathbf{Q} = \mathrm{Arr}_G (\mathbf{W},\mathbf{B})$ has a.s.\ non-identical exit rates, then the correlation $\rho = \rho (\mathbf{Q})$ satisfies
\begin{equation}\label{eq:main result rho}
\E (\rho) \geq 1 - 8\left(1 + O(n^{-1/2}) \right) \left( \frac{\sigma_B}{\sigma_W} \right)^2.
\end{equation}
\end{theorem}
\noindent The condition that $n \geq 6$ arises naturally in the proof, but it is not significant otherwise. We have not tried to optimize the coefficient $8(1+O(n^{-1/2}))$, but we expect that it can be improved, e.g., by incorporating information about the degree of $G$ (see \cref{rem:degree}). The condition that the exit rates are a.s.\ non-identical ensures that the correlation $\rho$ is well defined and amounts to a mild restriction on the dependence between the well depths and barrier heights.

\begin{figure}
    \centering
    \includegraphics[width=\linewidth]{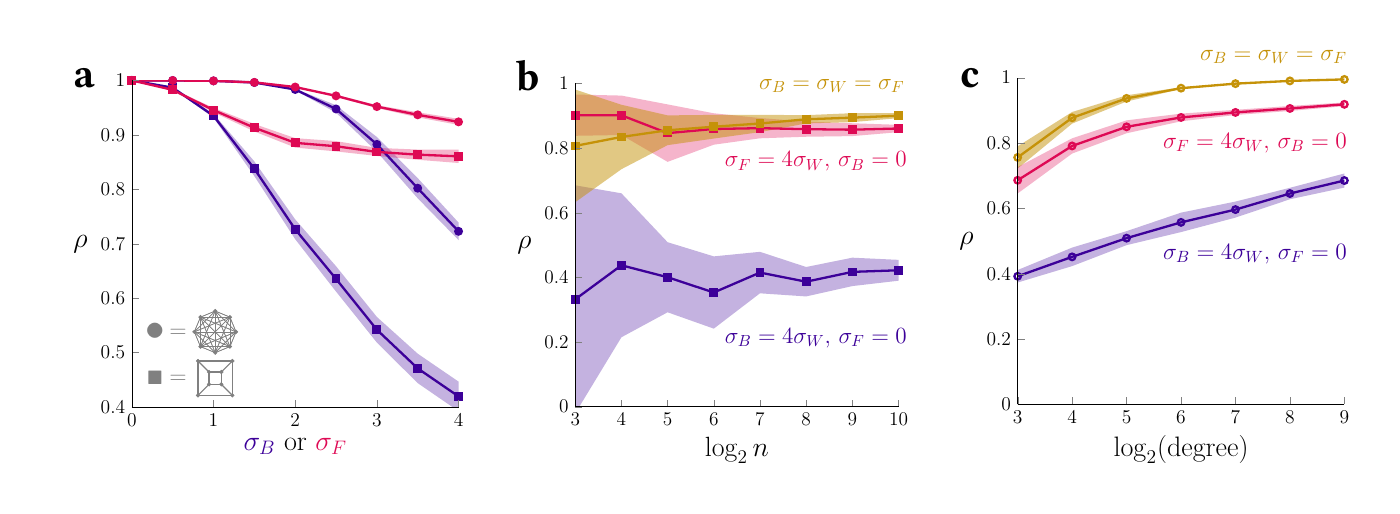}
    \caption{Trends in correlation. (\textbf{a}) Correlation decreases as $\sigma_B$ increases relative to $\sigma_W = 1$ with $\sigma_F = 0$ (blue curves), on both the complete graph (circles) and hypercube (squares) with $n = 2^{10}$ states. The same is true when the roles of $\sigma_B$ and $\sigma_F$ are reversed (red curves), although the correlation decreases more gradually in this case. (\textbf{b}) Regardless of whether the magnitudes of wells, barriers, and forces are balanced (yellow), forces dominate (red), or barriers dominate (blue), the correlation is relatively constant as the number $n$ of states increases. (\textbf{c}) However, increasing degree can favor higher correlation. In (\textbf{c}), random regular graphs with $2^{10}$ vertices were generated from a circulant graph of appropriate degree by swapping $10 \times \textrm{degree}$ pairs of edges uniformly at random, while keeping the graph simply connected. In all plots, $\sigma_W = 1$ and marks indicate the mean of $\rho$ over $25$ identical trials, with shaded regions corresponding to $\pm 1$ standard deviation.}
    \label{fig3}
\end{figure}

We expect that the proof of \cref{thm:rho} can be extended in several directions, including to non-regular and random adjacency graphs, non-Gaussian distributions on well depths and barrier heights, and energy landscapes with more general dependence among wells and barriers. A lower bound on the expected correlation exhibited by non-reversible Arrhenius-like dynamics, i.e., with nonzero forces in \cref{eq:rates_w_b_f}, would be especially valuable. The primary barrier to proving such a bound is the lack of a useful expression for the stationary distribution $\bs{\pi}$. We could extend \cref{thm:rho} to the case of weak forces using results from the literature on Markov chain perturbation (e.g., \cite{Mitrophanov2003}), but \cref{fig1,fig3} show that the correlation can also be high when forces are strong.

\subsection{Application to spin-glass dynamics}\label{subsec:application}

Our second main result is an application of \cref{thm:rho} to a family of dynamics associated with the random energy model (REM). Unlike the examples considered in \cref{fig1,fig2,fig3}, these dynamics entail barrier heights that depend on the depths of adjacent wells.

The REM is a mean-field model of a spin-glass consisting of $N$ spins, the configurations of which have random energies \cite{derrida_random-energy_1980,derrida_random-energy_1981}. More precisely, it is a random Boltzmann distribution on the hypercube $\{-1,1\}^N$ that arises from setting the well depths $W_i$ in \cref{eq:boltzmann} to be i.i.d.\ $\mathcal{N}(0,\sigma_W^2)$ random variables, where $\sigma_W^2 = \beta^2 N$ and $\beta > 0$ denotes inverse temperature. We consider dynamics of the REM that Bouchaud introduced to study the phenomenon of aging \cite{refId0}, a generalization of which is known as the Bouchaud trap model \cite{MR2581889}.

As before, we let $G$ be a connected, regular graph with $n$ vertices. For a parameter $\lambda \in [0,1]$ and well depths $\mathbf{W} = (W_i)_{i \in [n]}$, we denote by $\mathrm{REM}_G(\lambda, \mathbf{W})$ the transition rate matrix $\mathbf{R}$ with entries 
\begin{equation}\label{eq:rem_dynamics}
R_{ij} = \exp\left( \lambda W_i - (1-\lambda) W_j \right), \quad (i,j) \in E(G).
\end{equation}
Note that $\mathbf{R}$ satisfies detailed balance with respect to the Boltzmann stationary distribution \eqref{eq:boltzmann}, for every $\lambda \in [0,1]$.

By varying $\lambda$ from $0$ to $1$, the transition rates interpolate between two extremes. The first entails transitions from $i$ to $j \sim i$ that occur at a rate of $\exp(-W_j)$. In other words, transitions occur at a rate that depends on the energy of the destination. The second extreme entails rates of $\exp(W_i)$ that instead depend on the energy of the originating state, again for $j \sim i$. For this reason, $\lambda$ is said to measure the ``locality'' of the dynamics \cite{MR2581889}.

With this interpretation of $\lambda$, \cref{thm:rho} states the intuitive fact that, for the REM dynamics to exhibit high correlation between the effective potential and its local part, it suffices for the dynamics to be highly local. It further shows that it is enough for $\lambda$ to be close to $1$ in a way that does not depend on the underlying graph $G$ or any other model parameters. In the following statement, we use $\mathbf{I}_n$ to denote the $n \times n$ identity matrix.

\begin{theorem}\label{thm:rem}
    Let $G$ be a connected, regular graph with $n \geq 6$ vertices, let $\sigma_W > 0$, and let $\lambda \in [0,1]$. If $\mathbf{W} \sim \mathcal{N}(\mathbf{0},\sigma_W^2 \mathbf{I}_n)$, then the correlation $\rho = \rho (\mathbf{R})$ exhibited by the REM dynamics $\mathbf{R} = \mathrm{REM}_G (\lambda,\mathbf{W})$ satisfies
    \[
    \E (\rho) \geq 1 - 16 \left(1+O(n^{-1/2})\right) (1-\lambda)^2.
    \]
\end{theorem}

We emphasize the fact that the lower bound holds uniformly over all possible values of $\sigma_W^2$, hence all values of the inverse temperature $\beta$ as well. Considering that both the distribution and dynamics of the REM exhibit phase transitions as $\beta$ crosses a critical value of $\beta_c = \sqrt{2\log 2}$ \cite{PhysRevLett.88.087201}, it is interesting that $\beta$ plays no role in the correlation lower bound. This observation is reminiscent of the fact that the dynamical phase transition in the REM dynamics is not accompanied by a qualitative change in the spectral gap \cite{10.1214/aoap/1028903457}.

\begin{proof}[Proof of \cref{thm:rem}]
    We consider $\lambda \in [0,1)$, because the correlation trivially satisfies $\rho = 1$ a.s.\ when $\lambda = 1$. We can express the transition rates $R_{ij}$ in the Arrhenius-like form \eqref{eq:rates_w_b} by taking the barrier heights to be $B_{ij} = (1-\lambda) (W_i+W_j)$ for every $(i,j) \in E(G)$: 
    \[
    R_{ij} = \exp (\lambda W_i - (1-\lambda) W_j) = \exp (W_i - (1-\lambda) (W_i + W_j)) = \exp (W_i - B_{ij}).
    \]
    It is easy to verify that the corresponding exit rates $q_i = \sum_{j:\, j \sim i} R_{ij}$ are a.s.\ non-identical because the well depths $(W_i)_{i \in [n]}$ are i.i.d.\ $\mathcal{N}(0,\sigma_W^2)$ random variables. For the same reason, the barrier heights satisfy $B_{ij} = B_{ji} \sim \mathcal{N}(0,\sigma_B^2)$ with $\sigma_B = \sqrt{2}(1-\lambda)\sigma_W$. Furthermore, taking $f: \R \to \R$ to be $f(x) = (1-\lambda)x$ and $\sigma = \sigma_B/\sqrt{2}$, the barrier heights adjacent to $i$ satisfy
    \[
    (B_{ij}-f(W_i))_{j: \, j \sim i} = ((1-\lambda)W_j )_{j: \, j \sim i} \sim \mathcal{N}(\mathbf{0}, \sigma^2 \mathbf{I}_{d}),
    \]
    in terms of the degree $d$ of $G$. Since $B_{ij} - f(W_i)$ is independent of $W_i$ for every $i \in [n]$, $(\mathbf{W},\mathbf{B})$ is a separable, $(\sigma_W, \sigma_B)$-disordered energy landscape.  \cref{thm:rho} therefore implies that
    \[
    \E (\rho) \geq 1 - 8\left(1 + O(n^{-1/2}) \right) \left( \frac{\sigma_B}{\sigma_W} \right)^2 = 1 - 16 \left(1 + O(n^{-1/2}) \right) (1 - \lambda)^2.
    \]
\end{proof}

\subsubsection*{Organization}

The proof of \cref{thm:rho} has two main steps. The first step bounds below $\E (\rho)$ in terms of a  ratio of two variances, one concerning the well depths and the other concerning a function of the barrier heights (\cref{sec:proof_of_rho}). The second step calculates one variance and estimates the other using the Gaussian concentration inequality (\cref{sec:proof_of_r2}). In \cref{sec:proof of main result}, we combine the results from the two steps to prove \cref{thm:rho}.

\section{Relating the correlation to a ratio of variances}\label{sec:proof_of_rho}

The goal of this section is to prove the following proposition, which implements the first step of the proof of \cref{thm:rho}. We state it terms of a function $A_i$ of the barrier heights $(B_{ij})_{j : \, j \sim i}$, which plays an important role in the proof:
\begin{equation}\label{eq:def_of_A}
A_i = \log \sum_{j: \, j \sim i} \exp(-B_{ij}).
\end{equation}
    
\begin{proposition}\label{prop: first part}
    Let $G$ be a connected graph with $n \geq 2$ vertices. If $\mathbf{W} \in \R^n$ and $\mathbf{B} \in \R^{|E(G)|}$ are random vectors satisfying $B_{ij} = B_{ji}$ for all $(i,j) \in E(G)$ and such that the correlation $\rho = \rho (\mathbf{Q})$ of $\mathbf{Q} = \mathrm{Arr}_G (\mathbf{W},\mathbf{B})$ is a.s.\ well defined, then
    \begin{equation}\label{eq: first part bound}
        \E(\rho) \geq 1 - 2 \, \E \left( \frac{\Var A_I}{\Var W_I} \right),
    \end{equation}
    where $I \in [n]$ is uniformly random and independent of $\mathbf{W}$ and $\mathbf{B}$.
\end{proposition}

\cref{prop: first part} requires less of $G$, $\mathbf{W}$, and $\mathbf{B}$ than \cref{thm:rho}. In particular, $G$ does not need to be a regular graph, and the distributions of $\mathbf{W}$ and $\mathbf{B}$ are limited only by the requirement that $\rho$ is well defined. Note that $\rho$ is well defined if and only if both the stationary distribution $\bs{\pi}$ and exit rates $\mathbf{q}$ have non-identical entries. If $B_{ij} = B_{ji}$ for all $(i,j) \in E(G)$, then $\mathbf{Q}$ has the Boltzmann stationary distribution \eqref{eq:boltzmann}. In this case, $\bs{\pi}$ and $\mathbf{q}$ have non-identical entries if and only if $\mathbf{W}$ and $\mathbf{W} + \mathbf{A}$ do.

\begin{proof}[Proof of \cref{prop: first part}]
    For the moment, suppose that $\mathbf{W}$ and $\mathbf{B}$ are fixed vectors such that $B_{ij} = B_{ji}$ for every $(i,j) \in E(G)$ and the correlation $\rho$ exhibited by $\mathbf{Q} = \mathrm{Arr}_G(\mathbf{W},\mathbf{B})$ is well defined. We claim that 
    \begin{equation}\label{eq:rho_formula}
    \rho = \frac{1+\wh\rho r}{\sqrt{1+2\wh\rho r + r^2}},
    \end{equation}
    in terms of the quantities $\wh\rho$ and $r$, defined as
    \[
    \wh\rho = \Corr (W_I, A_I) \quad \text{and} \quad r = \sqrt{\frac{\Var A_I}{\Var W_I}}.
    \]
    
    Indeed, by the definition of the Arrhenius-like dynamics \eqref{eq:rates_w_b}, the exit rates of $\mathbf{Q}$ satisfy
    \[
    q_i = \sum_{j : \, j \sim i} Q_{ij} = e^{W_i} \sum_{j : \, j \sim i} e^{-B_{ij}} = e^{W_i + A_i}.
    \]
    Moreover, since $\mathbf{B}$ is symmetric, the stationary distribution of $\mathbf{Q}$ is the Boltzmann distribution $\pi_i \propto \exp(-W_i)$. The correlation therefore equals
    \[
    \rho = \Corr (-\log \pi_I, \log q_I) = \Corr (W_I, W_I + A_I).
    \]
    To verify \cref{eq:rho_formula}, we use the definition of correlation and linearity of covariance, and then identify factors of $\wh\rho$ and $r$ in the resulting expression: 
    \begin{align*}
    \Corr (W_I, W_I + A_I) &= \frac{\Var W_I + \Cov(W_I,A_I)}{\sqrt{(\Var W_I) ( \Var W_I + 2 \Cov (W_I, A_I) + \Var A_I)}}\\ 
    &= \frac{1+\wh\rho r}{\sqrt{1+2\wh\rho r + r^2}}.
    \end{align*}

    Next, let $\mathbf{W}$ and $\mathbf{B}$ be random vectors that satisfy the hypotheses. To complete the proof, we must show that 
    \begin{equation}\label{eq:rho 2r2}
    \E (\rho) \geq 1 - 2\, \E (r^2).
    \end{equation}
    Let $\rho_+ = \rho \1_{\{\rho \, \geq \, 0\}}$ denote  the nonnegative part of the correlation. We claim that
    \[
    \E (\rho_+) \geq 1 - \E (r^2) \quad \text{and} \quad \P(\rho < 0) \leq \E(r^2).
    \]
    These two bounds together imply \cref{eq:rho 2r2} because $\rho \geq -1$:
    \[
    \E (\rho) = \E (\rho_+) + \E (\rho \1_{\{\rho < 0\}}) \geq \E (\rho_+) - \P (\rho < 0) \geq 1 - 2 \, \E (r^2).
    \]
    We address the lower bound of $\E(\rho_+)$ first. 
    
    By assumption, $\rho$ is a.s.\ well defined, hence $\rho$ satisfies \cref{eq:rho_formula}. By \cref{eq:rho_formula}, when $r \leq 1$, $\rho$ is nonnegative and minimized by $\wh\rho = -r$, which implies the bound
    \[
    \rho_+^2 \geq 1-r^2.
    \]
    However, $r$ is random and can take values greater than $1$, so we must first condition on the occurrence of $\{r \leq 1\}$ to apply it. In terms of the probability $p = \P (r \leq 1)$, we find that
    \begin{equation}\label{eq:conditionalErho}
    \E (\rho_+^2) \geq \E (\rho_+^2 \mid r \leq 1) \, p \geq \left( 1 - \E ( r^2 \mid r \leq 1 ) \right) p.
    \end{equation}
    To bound above the conditional expectation of $r^2$, we note that
    \[
    \E (r^2) = \E (r^2 \mid r \leq 1) \, p + \E (r^2 \mid r > 1) (1-p) \geq \E (r^2 \mid r \leq 1) \, p + (1-p).
    \]
    Rearranging this expression gives
    \[
    \E (r^2 \mid r \leq 1) \leq \frac{1}{p}\big(\E (r^2) - (1-p)\big).
    \]
    We substitute this bound into \cref{eq:conditionalErho} to conclude that
    \[
    \E (\rho_+^2) \geq \left(1 - \frac{1}{p}\big(\E (r^2) - (1-p)\big) \right) p = 1 - \E (r^2).
    \]
    Since $\rho_+ \in [0,1]$, $\E(\rho_+)$ is at least $\E(\rho_+^2)$, which proves that $\E(\rho_+) \geq 1 - \E(r^2)$.

    To conclude, we note that the probability $\P(\rho < 0)$ is at most $\E(r^2)$ because 
    \[
    \P (\rho < 0) \leq \P (r > 1) = \P (r^2 > 1) \leq \E(r^2).
    \]
    The first inequality holds because \cref{eq:rho_formula} implies that $\rho \geq 0$ whenever $r \leq 1$. The second inequality is due to Markov's inequality.
\end{proof}

\section{Bounding the expected ratio of variances}\label{sec:proof_of_r2}

Recall that $A_i$ is the function of the barrier heights introduced in \cref{eq:def_of_A}, and $I$ is a uniformly random state that is independent of all other randomness. The goal of this section is to prove an upper bound on the expectation of $r^2 = (\Var A_I)/(\Var W_I)$. Substituting this bound into \cref{prop: first part} will imply \cref{thm:rho}.

Bounding above $\E(r^2)$ is difficult for two reasons. First, the quantities $\Var W_I$ and $\Var A_I$ are not necessarily independent. Second, while $\Var W_I$ is the sample variance of i.i.d.\ Gaussian random variables, which has a known distribution, $\Var A_I$ is an apparently complicated function of (possibly dependent) barrier heights. To address the first difficulty, because $\Var W_I$ is the simpler of the two quantities, we replace $\mathbf{W}$ with an i.i.d.\ copy $\mathbf{W}'$ and bound the associated error:
\begin{equation}\label{eq:split_r^2}
\E (r^2) = \E \left( \frac{\Var A_I}{\Var W_I}\right) = \E (\Var A_I) \cdot \E \left(\frac{1}{\Var W_I} \right) + \E \left( \Var A_I \left( \frac{1}{\Var W_I} - \frac{1}{\Var W_I'} \right) \right).
\end{equation}
The difference of reciprocals
\[
D_I = \frac{1}{\Var W_I} - \frac{1}{\Var W_I'}
\]
satisfies $\E (D_I^2) = 2 \V (1/\Var W_I)$, where $\V$ denotes variance with respect to all randomness, in contrast to $\Var (\cdot) = \V ( \cdot \mid (\mathbf{W},\mathbf{B}))$. Hence, by applying the Cauchy--Schwarz inequality to the term $\E ( (\Var A_I) D_I)$ in \cref{eq:split_r^2}, we find that
\begin{equation}\label{eq:r2_starting_point}
\E (r^2) \leq \E (\Var A_I) \cdot \E \left( (\Var W_I)^{-1} \right) + \sqrt{\, \E\big( (\Var A_I)^2 \big) \cdot 2 \,\V \big( (\Var W_I)^{-1} \big)}.
\end{equation}
The next two subsections address the factors involving $\Var W_I$ and $\Var A_I$, respectively. We combine them in the last subsection.

\subsection{The inverse distribution of \texorpdfstring{$\Var W_I$}{Var WI}}

In the statement and proof of the following lemma, we use $\mathrm{InvGam}(a,b)$ to denote the distribution of an inverse gamma random variable with shape and scale parameters $a$ and $b$, both positive real numbers. For reference, if $U \sim \mathrm{InvGam}(a,b)$, then the probability density function of $U$ is
\[
\frac{b^a}{\Gamma (a)} x^{-a-1} e^{-b/x}, \quad x \in (0,\infty),
\]
which implies that $\E (U) = b/(a-1)$ for $a > 1$ and $\V (U) = b^2/(a-1)^2 (a-2)$ for $a > 2$. The fact that $a > 2$ is required for $\V(U)$ to be finite underlies the condition in \cref{thm:rho} that the adjacency graph $G$ must have $n \geq 6$ vertices.

\begin{lemma}\label{lem:invvar}
Let $n \geq 6$ be an integer and $\sigma_W > 0$. If $\mathbf{W} \sim \mathcal{N}(\mathbf{0},\sigma_W^2 \mathbf{I}_n)$ and $I \in [n]$ is independent and uniformly random, then the distribution of $U = (\Var W_I)^{-1}$ is $\mathrm{InvGam}\big((n-1)/2,n/(2\sigma_W^2)\big)$. In particular, the mean and variance of $U$ are
    \[
    \E (U) = \frac{n}{(n-3) \sigma_W^2} \quad \text{and} \quad \V (U) = \frac{2n^2}{(n-3)^2 (n-5)\sigma_W^4}.
    \]
\end{lemma}

\begin{proof}
    Due to the independence of the well depths $W_i$ and because $I$ is uniformly random, the quantity $(n/(n-1))\Var W_I$ is the (unbiased) sample variance of $n$ i.i.d.\ $\mathcal{N}(0,\sigma_W^2)$ random variables. It is well known that scaling by $(n-1)/\sigma_W^2$ produces a chi-squared random variable with $n-1$ degrees of freedom:
    \[
    \frac{n-1}{\sigma_W^2} \cdot \frac{n}{n-1} \Var W_I \stackrel{d}{=} \chi_{n-1}^2.
    \]
    Using the scaling properties of gamma random variables, we find that
    \[
    U = \frac{1}{\Var W_I} \stackrel{d}{=} \frac{n}{\sigma_W^2 \chi_{n-1}^2} \stackrel{d}{=} \,\, \mathrm{InvGam} \left( \frac{n-1}{2}, \frac{n}{2\sigma_W^2} \right).
    \]
    The expressions for the expectation and variance of $U$ follow from the formulas preceding the statement of the lemma. In particular, we require $(n-1)/2 > 2$, hence $n \geq 6$, to ensure that the variance of $U$ is finite.
\end{proof}

\subsection{Bounds on the moments of \texorpdfstring{$\Var A_I$}{Var AI}}

Next, we consider the factors in \cref{eq:r2_starting_point} involving $\Var A_I$. According to the definition of a separable disordered energy landscape (\cref{def:separable}), there are $f: \R \to \R$ and $\sigma > 0$ such that, for every state $i \in [n]$, the shifted barrier heights $(B_{ij} - f(W_i))_{j: \, j \sim i}$ are i.i.d.\ $\mathcal{N}(0,\sigma^2)$ random variables. To use this property, we will denote $Z_i = f(W_i)$ and write $A_i$ as 
\begin{equation}\label{eq: a is lse}
A_i = \log \sum_{j: \, j \sim i} e^{-B_{ij}} = \log \sum_{j: \, j \sim i} e^{Z_i - B_{ij}} - Z_i.
\end{equation}
Expressing $A_i$ in this way is useful because the ``log-sum-exp'' function in \cref{eq: a is lse} is Lipschitz, as we elaborate below.

For every integer $d \geq 1$, we define the log-sum-exp function $\lse: \R^d \to \R$ by
\begin{equation}\label{eq:logsumexp}
\lse (\mathbf{x}) = \log \sum_{i \in [d]} e^{x_i}.
\end{equation}
Additionally, for a real number $L \geq 0$ and an integer $d \geq 1$, we say that $f: \R^d \to \R$ is $L$-Lipschitz (with respect to the $L^2$-norm) if 
\[
\forall \ \mathbf{x}, \mathbf{y} \in \R^d, \quad |f(\mathbf{x}) - f(\mathbf{y})| \leq \| \mathbf{x} - \mathbf{y}\|_2.
\] 

\begin{lemma}\label{lem:logsumexp partial bd}
For every integer $d \geq 1$, $\lse$ is $1$-Lipschitz with respect to the $L^2$ norm.
\end{lemma}

\begin{proof}
    By the intermediate value theorem, for every $\mathbf{x}, \mathbf{y} \in \R^d$, there exists some $\mathbf{z} \in \R^d$ such that $|\lse(\mathbf{x}) - \lse(\mathbf{y})| = |\nabla \lse(\mathbf{z}) \cdot (\mathbf{x} - \mathbf{y})|$. The partial derivatives of $\lse$ lie in $[0,1]$ and sum to $1$, so $\|\nabla \lse(\mathbf{z})\|_2 \leq 1$. The Cauchy--Schwarz inequality then implies that
    \[
    |\lse(\mathbf{x}) - \lse(\mathbf{y})| \leq \| \nabla \lse (\mathbf{z})\|_2 \, \|\mathbf{x} - \mathbf{y}\|_2 \leq \|\mathbf{x} - \mathbf{y}\|_2.
    \]
\end{proof}

Lipschitz functions of i.i.d.\ Gaussians exhibit dimension-free concentration according to the Gaussian concentration inequality (e.g., \cite[Theorem~5.2.2]{Vershynin_2018}). Being a standard result, we state it without proof.

\begin{lemma}\label{lem:gci}
    Let $d \geq 1$ be an integer and $\sigma > 0$. If $\g \sim \mathcal{N}(\mathbf{0},\sigma^2 \mathbf{I}_d)$ and $f: \R^d \to \R$ is $1$-Lipschitz, then
    \[
    \P (|f(\g) - \E f(\g)| \geq t) \leq 2 e^{-t^2/(2\sigma^2)}, \quad \quad t \geq 0.
    \]
\end{lemma}

We will use \cref{lem:gci} to obtain bounds on the central moments of $A_1$ that do not depend on the dimension $d$. In our setting, the dimension is $d-1$, where $d$ is the degree of the adjacency graph $G$. This partly explains why the conclusion of \cref{thm:rho} holds regardless of $d$.

\begin{lemma}\label{lem:vargci}
    Let $d \geq 1$ be an integer, let $\sigma > 0$, and let $f: \R^d \to \R$ be $1$-Lipschitz. If $\g \sim \mathcal{N}(\mathbf{0},\sigma^2 \mathbf{I}_d)$, then, for every integer $m \geq 1$, the $m$th central moment of $f(\g)$ satisfies
    \[
    \mu_m (f(\g)) = \E \left( | f(\g) - \E f(\g) |^m \right) \leq m \, \Gamma (m/2) (2\sigma^2)^{m/2}.
    \]
\end{lemma}

\begin{proof}
    By the tail integral formula and \cref{lem:gci}, the central moments of $f$ satisfy
    \[
    \mu_m (f (\g)) = \int_0^\infty \P \left( |f(\g) - \E f(\g)|^m \geq t \right) dt \leq \int_0^\infty 2 e^{-t^{2/m}/(2 \sigma^2)} dt = m \, \Gamma (m/2) (2\sigma^2)^{m/2}.
    \]
\end{proof}

While we could use \cref{lem:vargci} to bound the central moments of $A_1$, these moments are not the same as the quantities $\E (\Var A_I)$ and $\E ( (\Var A_I)^2)$ that appear in \cref{eq:r2_starting_point}. Intuitively, the variance of $A_I$ over a fixed landscape should be closely related to the variance of $A_1$ over different realizations of the landscape, because the distribution of barrier heights adjacent to a state $i$ is the same for every $i$, and $I$ is independent of the landscape. This is the idea behind the proof of the next result, which will allow us to bound the first two moments of $\Var A_I$ using the second and fourth central moments of $A_1$. We accept a crude bound on the second moment of $\Var A_I$, because its coefficient will matter little to our main result.

\begin{lemma}\label{lem:replace_var}
Let $n \geq 1$ be an integer. If $\mathbf{X} = (X_i)_{i \in [n]} \in \R^n$ is a vector of identically distributed, square-integrable random variables, and if $I \in [n]$ is uniformly random, independently of $\mathbf{X}$, then 
\begin{equation}\label{eq:replace_var_bds}
\E (\Var X_I) \leq \V (X_1) \quad \text{and} \quad \E \big( (\Var X_I)^2 \big) \leq 40 \mu_4 (X_1).
\end{equation}
\end{lemma}

\begin{proof}
    We prove the first bound in \cref{eq:replace_var_bds} by applying the law of total variance to the variance of $X_I$. Since $X_I$ is square-integrable, the law of total variance implies both
    \begin{align}
    \V (X_I) &= \E \big( \V (X_I \,|\, \X) \big) + \V \big( \E (X_I \,|\, \X) \big) \quad \text{and} \label{eq:lotv1}\\ 
    &= \E \big(\V (X_I \,|\, I)\big) + \V \big(\E (X_I \,|\, I) \big).\label{eq:lotv2}
    \end{align}
    Note that the quantity $\E (\V (X_I \,|\, \X))$ in \cref{eq:lotv1} is the one the first bound concerns, which we previously denoted as $\E (\Var X_I)$. Since $I$ is independent of the $X_i$, which have identical distributions, the terms on the right-hand side of \cref{eq:lotv2} satisfy
    \[
    \E \big( \V (X_I \mid I) \big) = \V (X_1) \quad \text{and} \quad \V \big( \E (X_I \mid I) \big) = 0.
    \]
    The first bound in \cref{eq:replace_var_bds} follows from substituting the preceding equalities into \cref{eq:lotv2} and then by using the nonnegativity of variance:
    \[
    \E \big( \V (X_I \mid \X) \big) = \V (X_1) - \V \big( \E (X_I \mid \X) \big) \leq \V (X_1).
    \]

    We turn to proving the second bound in \cref{eq:replace_var_bds}. We first note that, because $I$ is uniformly random, we can write $(\Var X_I)^2$ as
    \begin{equation}\label{eq:var2}
    (\Var X_I)^2 = \frac{1}{n^2} \sum_{i \in [n]} (X_i - \bar{X})^4 + \frac{2}{n^2} \sum_{1 \leq i < j \leq n} (X_i - \bar{X})^2 (X_j - \bar{X})^2,
    \end{equation}
    where $\bar{X}$ denotes the arithmetic mean $\frac{1}{n} \sum_{i \in [n]} X_i$. Since the $X_i$ are identically distributed, by applying the Cauchy--Schwarz inequality to the summands of the second term in \cref{eq:var2}, we find that
    \[
    \E \left( (\Var X_I)^2 \right) \leq \frac{1}{n^2} \sum_{i \in [n]} \E \left((X_i - \bar{X})^4\right) + \frac{2}{n^2} \sum_{1 \leq i < j \leq n} \E \left((X_i - \bar{X})^4\right) = \E \left( (X_1 - \bar{X})^4 \right).
    \]
    This is not quite the second bound in \cref{eq:replace_var_bds}, because $\bar{X}$ is not the same as its mean $\mu = \E X_1$. We substitute $X_i - \bar{X} = (X_i - \mu) + (\mu - \bar{X})$ into the preceding bound and use the inequality
    \begin{equation}\label{eq: abpower}
    |a+b|^k \leq 2^{k-1} (|a|^k + |b|^k), 
    \end{equation}
    which holds for all real numbers $a$, $b$, and $k \geq 1$, to find that
    \[
    \E \left( (\Var X_I)^2 \right) \leq 8 \left( \E \left( (X_1 - \mu)^4 \right) + \E \left( (\bar{X} - \mu)^4 \right) \right) = 8 (\mu_4 (X_1) + \mu_4 (\bar{X})).
    \]
    To address the second term on the right-hand side, we note that $\bar{X} - \mu = \frac{1}{n} \sum_{i \in [n]} Y_i$, where $Y_i = X_i - \mu$. Using the multinomial theorem, we can expand $\mu_4 (\bar{X})$ as
    \[
    \mu_4 (\bar{X}) = \frac{1}{n^4} \sum_{\substack{k_1+k_2+\, \cdots\, + k_n = 4\\ k_1, k_2, \,\dots,\, k_n \geq 0}} \binom{4}{k_1, k_2, \dots, k_n} \, \E \Bigg(\prod_{l \in [n]} Y_l^{k_l} \Bigg).
    \]
    Since the $Y_i$ are identically distributed, we can crudely bound expectations of their products by
    \[
    \E \Bigg(\prod_{l \in [n]} Y_l^{k_l} \Bigg) \leq 4 \, \E (Y_1^4) = 4 \, \E \left( (X_1 - \mu)^4 \right) = 4\, \mu_4 (X_1).
    \]
    Lastly, because the multinomial coefficients sum to $n^4$, the preceding bounds together imply that
    \[
    \E \left( (\Var X_I)^2 \right) \leq 8 (\mu_4 (X) + \mu_4 (\bar{X})) \leq 40\, \mu_4 (X_1).
    \]
\end{proof}

\begin{remark}\label{rem:degree}
   Although the estimates of \cref{thm:rho,thm:rem} do not depend on the degree of $G$, the simulations in \cref{fig2,fig3} show that higher degree favors higher correlation. It should be possible to improve the lower bounds in our main results with degree-dependent terms by specializing the first bound in \cref{eq:replace_var_bds} to the case when $X_i = A_i (\mathbf{B})$. However, improving \cref{eq:replace_var_bds} in this case appears to require a good lower bound on the covariance of $A_i (\mathbf{B})$ and $A_j (\mathbf{B})$ for $j \sim i$.
\end{remark}

We now combine the preceding lemmas to establish upper bounds on the first two moments of $\Var A_I$. As in \cref{lem:replace_var}, the coefficient in the bound on the second moment matters little to \cref{thm:rho,thm:rem}.

\begin{lemma}\label{lem: final varai bds}
    Let $\sigma_W$ and $\sigma_B$ be positive real numbers, and let $G$ be a regular graph. If $(\mathbf{W},\mathbf{B})$ is a separable, $(\sigma_W,\sigma_B)$-disordered energy landscape on $G$, then the variance of $A_I = A_I (\mathbf{B})$ with respect to an independent, uniformly random $I \in V(G)$ satisfies
    \[
    \E (\Var A_I) \leq 4 \sigma_B^2 \quad \text{and} \quad \E \left( (\Var A_I)^2 \right) \leq 1720 \sigma_B^4.
    \]
\end{lemma}

\begin{proof}
    We aim to apply \cref{lem:replace_var} with $X_i = A_i (\mathbf{B})$, which requires the $A_i (\mathbf{B})$ to be identically distributed and $A_I (\mathbf{B})$ to be square-integrable. To verify the first condition, we recall \cref{eq: a is lse}. It states that, because $(\mathbf{W},\mathbf{B})$ is a separable energy landscape, there are $f: \R \to \R$ and $\sigma > 0$ such that 
    \[
    A_i (\mathbf{B}) = \lse (\mathbf{b}_{i}) - Z_i,
    \]
    in terms of a vector $\mathbf{b}_{i} = (Z_i - B_{ij})_{j: \, j \sim i}$ of i.i.d.\ $\mathcal{N}(0,\sigma^2)$ random variables and $Z_i = f(W_i)$, which is independent of $\mathbf{b}_i$. In particular, the conditional distribution of $A_i (\mathbf{B})$ given~$W_i$ is the same for every $i \in V(G)$. Since the $W_i$ are identically distributed, so too are the $A_i (\mathbf{B})$.
    
    Next, we note that $A_I (\mathbf{B})$ is square-integrable because $\lse (\mathbf{b}_i)$ and $Z_i$ are square-integrable for every $i \in V(G)$. Indeed, it is well known that, if $\mathbf{x} \in \R^d$, then $\lse (\mathbf{x})$ lies between $\max_{k \in [d]} x_k$ and $\max_{k \in [d]} x_k + \log d$. Hence, if $d$ is the degree of $G$, then the Bhatia--Davis inequality implies that the variance of $\lse (\mathbf{b}_i)$ cannot be larger than $(\log d)^2$. Regarding $Z_i$, since  $B_{ij} \sim \mathcal{N}(0,\sigma_B^2)$ for every $j \sim i$ and $Z_i$ is independent of $\lse (\mathbf{b}_i)$, the variances $s_i \equiv s_{ij} = \V(Z_i - B_{ij})$ and $t_i = \V(Z_i)$ satisfy $\sigma_B^2 = s_i + t_i$. In particular, the variance of $Z_i$ is at most $\sigma_B^2$.

    Using \cref{lem:replace_var} and the independence of $\mathbf{b}_1$ and $Z_1$, we find that $\Var A_I$ satisfies
    \[
    \E (\Var A_I (\mathbf{B})) \leq \V (A_1 (\mathbf{B})) = \V (\lse (\mathbf{b}_{1})) + \V (Z_1).
    \]
    Since $\lse$ is $1$-Lipschitz (\cref{lem:logsumexp partial bd}), an application of \cref{lem:vargci} with $m = 2$ shows that $\V (\lse (\mathbf{b}_1)) \leq 4s_1$, hence
    \[
    \E (\Var A_I (\mathbf{B})) \leq 4s_1 + t_1 \leq  4\sigma_B^2.
    \]
    The second inequality follows from the fact that $\sigma_B^2 = s_1 + t_1$.

    Concerning the second moment of $\Var A_I (\mathbf{B})$, \cref{lem:replace_var} states that 
    \[
    \E \big( (\Var A_I (\mathbf{B}))^2 \big) \leq 40 \, \mu_4 (A_1 (\mathbf{B})) = 40 \big(\mu_4 (\lse (\mathbf{b}_1)) + 6 \V(\lse (\mathbf{b}_1)) \V (Z_1) + \mu_4 (Z_1) \big).
    \]
    Here, we have used the fact that, if $S$ and $T$ are independent random variables with finite fourth moments, then 
    \[
    \mu_4 (S+T) = \mu_4 (S) + 6 \V(S) \V(T) + \mu_4 (T).
    \]
    Applying the same fact with $S = Z_1$ and $T = B_{1j} - Z_1$ shows that 
    \[
    \mu_4 (Z_1) = \mu_4 (B_{1j}) - 6 \V(Z_1) \V (B_{1j} - Z_1) - \mu_4 (B_{1j} - Z_1) \leq \mu_4 (B_{1j}) = 3\sigma_B^4.
    \]
    By \cref{lem:vargci} (with $m=4$), the fourth central moment of $\lse (\mathbf{b}_1)$ is at most $16 \sigma_B^4$. We use the simple bounds $\V(Z_1) \leq \sigma_B^2$ and $\V(\lse (\mathbf{b}_1)) \leq 4 \sigma_B^2$ from above to conclude that
    \[
    \E \big( (\Var A_I (\mathbf{B}))^2 \big) \leq 40 (3 \sigma_B^4 + 6 \sigma_B^2 \cdot 4 \sigma_B^2 + 16 \sigma_B^4) = 1720 \sigma_B^4.
    \]
\end{proof}

\subsection{Bound on the expected ratio of variances}

We are now ready to prove the main result of this section, which is an upper bound on $\E (r^2)$. The statement uses the same hypotheses as \cref{thm:rho}, except the requirement that the exit rates are a.s.\ non-identical.

\begin{proposition}\label{prop:r2}
    Let $G$ be a regular graph with $n \geq 6$ vertices, and let $\sigma_W$ and $\sigma_B$ be positive real numbers. If $(\mathbf{W},\mathbf{B})$ is a separable, $(\sigma_W,\sigma_B)$-disordered energy landscape on $G$, then the ratio $r = r(\mathbf{Q})$ associated with the Arrhenius-like dynamics $\mathbf{Q} = \mathrm{Arr}_G (\mathbf{W},\mathbf{B})$ satisfies
    \begin{equation}\label{eq: prop bound on r2}
    \E (r^2) \leq 4\left(1 + O(n^{-1/2}) \right) \left(\frac{\sigma_B}{\sigma_W}\right)^2.
    \end{equation} 
\end{proposition}

\begin{proof}
Our starting point is \cref{eq:r2_starting_point}:
\begin{equation*}
\E (r^2) \leq \E (\Var A_I) \cdot \E \left( (\Var W_I)^{-1} \right) + \sqrt{\, \E\big( (\Var A_I)^2 \big) \cdot 2 \,\V \big( (\Var W_I)^{-1} \big)}.
\end{equation*}
The hypotheses on $\mathbf{W}$ and $\mathbf{B}$ are precisely those required for both \cref{lem:invvar,lem: final varai bds} to apply.\footnote{Technically, the proof of \cref{lem: final varai bds} does not require the barriers to be symmetric.} We substitute the expressions for the moments of $(\Var W_I)^{-1}$ from \cref{lem:invvar}, as well as the upper bounds on the moments of $\Var A_I$ from \cref{lem: final varai bds}, to find that
\[
\E (r^2) \leq 4 \sigma_B^2 \cdot \frac{n}{(n-3)\sigma_W^2} + \sqrt{ 1720\, \sigma_B^4 \cdot \frac{4n^2}{(n-3)^2 (n-5)\sigma_W^4}} = 4 \big(1 + O(n^{-1/2})\big) \left( \frac{\sigma_B}{\sigma_W}\right)^2.
\]
\end{proof}

\section{Combining the propositions}\label{sec:proof of main result}

We now combine \cref{prop: first part,prop:r2} to prove \cref{thm:rho}. Recall that we are given a connected, regular graph $G$ with $n \geq 6$ vertices, and positive real numbers $\sigma_W$ and $\sigma_B$. We must show that, if $(\mathbf{W},\mathbf{B})$ is a separable, $(\sigma_W,\sigma_B)$-disordered energy landscape for which the transition rate matrix $\mathbf{Q} = \mathrm{Arr}_G(\mathbf{W},\mathbf{B})$ has a.s.\ non-identical exit rates, then the correlation $\rho = \rho(\mathbf{Q})$ satisfies
\[
\E (\rho) \geq 1 - 8 \left(1+O(n^{-1/2}) \right) \left( \frac{\sigma_B}{\sigma_W} \right)^2.
\]

\begin{proof}[Proof of \texorpdfstring{\cref{thm:rho}}{Theorem 1}]
Assume for the moment that $\rho$ is a.s.\ well defined. Then, according to \cref{prop: first part}, $\rho$ satisfies
\[
\E (\rho) \geq 1 - 2 \, \E \left( \frac{\Var A_I}{\Var W_I} \right),
\]
where $I \in [n]$ is uniformly random and independent of $(\mathbf{W},\mathbf{B})$. By \cref{prop:r2}, which has the same hypotheses \cref{thm:rho}, the expected ratio of variances satisfies 
\[
    \E \left( \frac{\Var A_I}{\Var W_I} \right) \leq 4\left(1 + O(n^{-1/2}) \right) \left(\frac{\sigma_B}{\sigma_W}\right)^2.
\]
The preceding inequalities together imply the claimed lower bound of $\E (\rho)$.

To finish the proof, we explain why $\rho$ is a.s.\ well defined. For $\rho$ to be well defined, $\bs{\pi}$ must be well defined, and $\bs{\pi}$ and $\mathbf{q}$ must have non-identical entries. The former is a.s.\ true because $\mathbf{Q}$ is a.s.\ irreducible, due to the fact that $G$ is connected and $Q_{ij} = \exp (W_i - B_{ij})$ is a.s.\ positive for every $(i,j) \in E(G)$. The latter is true because $\bs{\pi}$ is the Boltzmann distribution \eqref{eq:boltzmann} with i.i.d.\ Gaussian well depths, which are a.s.\ non-identical, while $\mathbf{q}$ has a.s.\ non-identical entries by assumption.
\end{proof}

\subsubsection*{Acknowledgements} 

J.C.\ and D.R.\ were supported by NSF award CCF-2106687 and the US Army Research Office Multidisciplinary University Research Initiative award W911NF-19-1-0233.

\bibliographystyle{alpha}

\newcommand{\etalchar}[1]{$^{#1}$}

\end{document}